\documentclass{aic}

\usepackage{amsthm,amsmath,mathtools}
\usepackage[capitalize,english]{cleveref}
\usepackage[abbrev]{amsrefs}
\usepackage{doi}

\renewcommand{\PrintDOI}[1]{\doi{#1}}

\newtheorem{THM}{Theorem}
\newtheorem{LEM}[THM]{Lemma}
\newtheorem{PROB}[THM]{Problem}
\theoremstyle{definition}
\newtheorem{EX}[THM]{Example}

\newcommand{\N}{\mathbb{N}}
\newcommand{\R}{\mathbb{R}}
\newcommand{\sub}{\subseteq}
\def\sm{\smallsetminus}
\newcommand{\braces}[1]{\left(#1\right)}
\newcommand{\menge}[1]{\left\{#1\right\}}
\newcommand{\abs}[1]{\left |#1\right |}
\newcommand{\tn}[1]{\textnormal{#1}}

\aicAUTHORdetails{%
  title = {Tangles are Decided by Weighted Vertex Sets}, 
  author = {Christian Elbracht, Jay Lilian Kneip, and Maximilian Teegen},
  plaintextauthor = {Christian Elbracht, Jay Lilian Kneip, Maximilian Teegen},
}   

\aicEDITORdetails{%
   year={2020},
   number={9},
   received={21 February 2019},   
   published={24 July 2020},  
   doi={10.19086/aic.13691}    
}   

\begin{document}

\begin{frontmatter}[classification=text]


\author[ce]{Christian Elbracht}
\author[jk]{Jay Lilian Kneip}
\author[mt]{Maximilian Teegen}

\begin{abstract}
	We show that, given a $ k $-tangle $ \tau $ in a graph $ G $, there always exists a weight function $ w\colon V(G)\to\N $ such that a separation $ (A,B) $ of $ G $ of order $ <k $ lies in $ \tau $ if and only if $ w(A)<w(B) $, where $ w(U)\coloneqq\sum_{u\in U}w(u) $ for $ U\sub V(G) $. We show that the same result holds also for tangles of hypergraphs as well as for edge-tangles of graphs, but not for edge-tangles of hypergraphs.
\end{abstract}
\end{frontmatter}


\section{Introduction}

Tangles in graphs have played a central role in graph minor theory ever since their introduction by Robertson and Seymour in~\cite{GMX}. Formally, a tangle in a graph $ G $ is an orientation of all low-order separations of $ G $ satisfying certain consistency assumptions. Tangles can be used to locate, and thereby capture the essence of, highly connected substructures in $ G $ in that every such substructure defines a tangle in $ G $ by orienting each low-order separation of $ G $ towards the side containing most or all of that substructure. In view of this, if some tangle in $ G $ contains the separation $ (A,B) $, we  think of $ A $ and $ B $ as the `small' and the `big' side of $ (A,B) $ in that tangle, respectively. Our main result will make this intuition concrete.

As a concrete example, if $ G $ contains an $ n\times n $-grid for large $ n $, then the vertex set of that grid defines a tangle $ \tau $ in $ G $ as follows.
Take note that no separation of low order can divide the grid into two parts of roughly equal size:
If the grid is large enough then at least 90\% of its vertices, say, will lie on the same side of such a separation.
Orienting towards that side all the separations of order $<k$ for some fixed $k$ much smaller than $n$ then gives a tangle $\tau$.
In this way, the vertex set of the $ n\times n $-grid `defines $ \tau $ by majority vote'.

In~\cite{ProfilesNew} Diestel raised the question whether all tangles in graphs arise in the above fashion, that is, whether all graph tangles are decided by majority vote by some subset of the vertices:
\begin{PROB}\label{prob:decider}
Given a $ k $-tangle~$ \tau $ in a graph $ G $, is there always a set $ X $ of vertices such that a separation $ (A,B) $ of order $ <k $ lies in $ \tau $ if and only if $ \abs{A\cap X}<\abs{B\cap X} $?
\end{PROB}
A partial answer to this was given in~\cite{ChristiansMasterarbeit}, where Elbracht showed that such a set $ X $ always exists if $ G $ is $ (k-1) $-connected and has at least $ 4(k-1) $ vertices. However Elbracht's approach relies heavily on the $ (k-1) $-connectedness of the graph and offers no line of attack for the general problem. Finding an answer for arbitrary graphs appears to be hard.

If a tangle in $ G $ is decided by some vertex set $ X $ by majority vote, this set $ X $ can be used as an oracle for that tangle, allowing one to store complete information about the complex structure of a tangle using a set of size at most $ \abs{V} $. On the other hand, if there were tangles without such a decider set, this would mean that tangles are a fundamentally more general concept than concrete highly cohesive subsets, not just an indirect way of capturing them.

In this paper, we consider a fractional version of Diestel's question and answer it affirmatively, making precise the notion that $ B $ is the `big' side of a separation $ (A,B)\in\tau $: given a $ k $-tangle $ \tau $ in $ G $, rather than finding a vertex set $ X $ which decides $ \tau $ by majority vote, we find a weight function $ w\colon V(G)\to\N $ on the vertices such that for all separations $ (A,B) $ of order $ <k $ we have $ (A,B)\in \tau $ if and only if the vertices in $ B $ have higher total weight than those in~$ A $.

Thus we show that every graph tangle is decided by some {\em weighted} set of vertices. This weight function, or weighted set of vertices, can then serve as an oracle for that tangle in the same way that a vertex set deciding the tangle by majority vote would. For any tangle, the existence of such a weight function with values in~$ \{0,1\} $ is equivalent to the existence of a vertex set $ X $ deciding that tangle by majority vote.

In~\cref{sec:weighted-deciders} of this paper we will formally define separations and tangles, and formulate and prove our main theorem asserting that tangles of graphs (and of hypergraphs) always admit such a weight function. Following that we show in~\cref{sec:edgetangles} that the same arguments are also applicable to edge-tangles of graphs, a relative of the tangles usually considered, and prove our main result also for this type of~tangle.

For settings beyond graphs it is known that the analogue of Diestel's question may be false. For instance, Geelen~\cite{GeelenExample} pointed out that there are matroid tangles which cannot be decided by majority vote, not even when considering a fractional version of the problem. For edge-tangles as analysed in~\cref{sec:edgetangles} the fractional version of~\cref{prob:decider} is true for graphs but may fail for hypergraphs. We demonstrate the latter with a counterexample, which, though discovered independently, is conceptually similar to Geelen's example in the matroid setting.

\section{Weighted deciders}\label{sec:weighted-deciders}

Formally, a {\em separation} of a graph $ G=(V,E) $ is a pair $ (A,B) $ with $ A\cup B=V $ such that~$ G $ contains no edge between $ A\sm B $ and $ B\sm A $, and the {\em order} of a separation $ (A,B) $ is the size $ \abs{A\cap B} $ of its separator $ A\cap B $. Furthermore, for an integer $ k $, a \emph{$ k $-tangle} in $ G $ is a set consisting of exactly one of $ (A,B) $ and $ (B,A) $ for every separation $ (A,B) $ of $ G $ of order $ <k $, with the additional property that no three `small' sides of separations in $ \tau $ cover $ G $, that is, that there are no $ (A_1,B_1),(A_2,B_2),(A_3,B_3)\in\tau $ for which $ G=G[A_1]\cup G[A_2]\cup G[A_3] $.

Our main result is the following:

\begin{THM}\label{thm:tangledecider}
	Let $ G=(V,E) $ be a finite graph and $ \tau $ a $ k $-tangle in $ G $. Then there exists a function $ w\colon V\to\N $ such that a separation $ (A,B) $ of $ G $ of order $ <k $ lies in $ \tau $ if and only if~$ w(A)<w(B) $, where $ w(U)\coloneqq\sum_{u\in U}w(u) $ for $ U\sub V $.
\end{THM}

We shall prove Theorem~\ref{thm:tangledecider} in the remainder of this section. Our general strategy will be as follows: we define a partial order on the separations of~$ G $ and consider the set of those separations of the $ k $-tangle~$ \tau $ that are maximal in this partial order. For these separations we will be able to show that, on average, their separators divide each other so that they lie more on the `big' side of each other, where `big' is the big side according to~$ \tau $. This will enable us to use a result from linear programming to find a weight function assigning weights to the vertices of these separators such that this weight function decides all these maximal separations of~$ \tau $ correctly. The nature of the partial order will then ensure that this weight function in fact decides all separations in~$ \tau $ correctly.

For a graph $ G $ there is a partial order on the separations of $ G $ given by letting $ {(A,B)\le(C,D)} $ if and only if $ A\sub C $ and $ B\supseteq D $. One of the main ingredients for the proof of Theorem~\ref{thm:tangledecider} is the following observation about those separations in a tangle $ \tau $ that are maximal in $ \tau $ with respect to this partial order.
It says, roughly, that they divide each other's separators so that, on average, those separators lie more on the big side of the separation than on the small side, according to the tangle.

\begin{LEM}\label{lem:cross-counting}
	For every $ k $-tangle $ \tau $ in a graph $ G $ and distinct maximal elements $ (A,B),(C,D) $ of $ \tau $ we have
	$ \abs{B\cap(C\cap D)}+\abs{D\cap(A\cap B)}>\abs{A\cap(C\cap D)}+\abs{C\cap(A\cap B)}$.
\end{LEM}

\begin{proof}
	Let $ \tau $ be a $ k $-tangle in $ G=(V,E) $ and $ (A,B) $ and $ (C,D) $ distinct maximal elements of~$ \tau $. Observe that $ (A\cup C\,,\,B\cap D) $ is a separation of $ G $ as well. In fact this separation is the supremum of $ (A,B) $ and $ (C,D) $ in the partial order. Therefore $ \tau $ cannot contain $ (A\cup C\,,\,B\cap D) $ by the assumed maximality of $ (A,B) $ and $ (C,D) $ in~$ \tau $. On the other hand $ \tau $ cannot contain $ (B\cap D\,,\,A\cup C) $ either since $ A $, $ C $, and $ B\cap D $ together cover~$ G $. Consequently, since $ \tau $ is a $ k $-tangle, we must have~$ \abs{(A\cup C)\cap(B\cap D)}\ge k $.
	
	Recall that $ \abs{A\cap B}<k $ and $ \abs{C\cap D}<k $ since $ \tau $ is a $ k $-tangle. Observe additionally that the order of separations is modular, that is,
	\[ \abs{A\cap B}+\abs{C\cap D}=\abs{(A\cup C)\cap(B\cap D)}+\abs{(A\cap C)\cap(B\cup D)}\,. \]	
        With the above inequalities this implies that~$ \abs{(A\cap C)\cap(B\cup D)}<k $, and hence in particular that \[ \abs{(A\cap C)\cap(B\cup D)}<\abs{(A\cup C)\cap(B\cap D)}. \] Adding~$ \abs{A\cap B\cap C\cap D} $ to both sides proves the claim. 
\end{proof}

Additionally we shall use a result from linear programming: Tucker's~Theorem, a close relative of the Farkas~Lemma.
For a vector $ x\in\R^n $ we use the usual shorthand notation $ x\ge 0 $ to indicate that all entries of $ x $ are non-negative, and similarly write $ x > 0 $ if all entries of $ x $ are strictly greater than zero.

\begin{LEM}[Tucker's Theorem~\cite{TuckerDuality}]\label{lem:tucker}
	Let $ K\in\R^{n\times n} $ be a skew-symmetric matrix, i.e. $ {K^T=-K} $. Then there exists a vector $ x\in\R^n $ such that
	\[ Kx\ge0 \quad \text{and}\quad x\ge 0 \quad \text{and}\quad x+Kx>0. \]
\end{LEM}

We are now ready to prove Theorem~\ref{thm:tangledecider}.

\begin{proof}[Proof of Theorem~\ref{thm:tangledecider}]
	Let a finite graph $ G=(V,E) $ and a $ k $-tangle $ \tau $ in $ G $ be given. Since $ G $ is finite it suffices to find a weight function $ {w\colon V\!\to\R_{\ge 0}} $ such that a separation $ (A,B) $ of order $ <k $ lies in $ \tau $ precisely if $ {w(A)<w(B)} $; by the density of the rationals in the reals, this $ w $ can then be turned into such a weight function with values in~$ \N $. 
	
	For this it is enough to find a function $ w\colon V\!\to\R_{\ge 0} $ such that $ w(A)<w(B) $ for all maximal elements $ (A,B) $ of $ \tau $: for if $ w(A)<w(B) $ and $ (C,D)\le(A,B) $ then
	\[ w(C)\le w(A)<w(B)\le w(D). \]
	So let us show that such a weight function $ w $ exists.
	
	To this end let $ (A_1,B_1),\dots,(A_n,B_n) $ be the maximal elements of $ \tau $ and set
	\[ m_{ij}\coloneqq\abs{B_i\cap(A_j\cap B_j)}-\abs{A_i\cap(A_j\cap B_j)} \]
	for $ i,j\le n $. Let $ M $ be the matrix $ \{m_{ij}\}_{i,j\le n} $. Observe that, by Lemma~\ref{lem:cross-counting}, we have $ m_{ij}+m_{ji}>0 $ for all $ i\ne j $ and hence the matrix $ M+M^T $ has positive entries everywhere but on its diagonal (where it has zeros). We further define
	\[ K'\coloneqq\frac{M+\mathrlap{M^T}\phantom{M}}{2}\;\qquad\tn{ and }\qquad K\coloneqq M-K'. \]
	Then $ K $ is skew-symmetric, that is, $ K^T=-K $. Let $ x=(x_1,\dots,x_n)^T $ be the vector obtained by applying Lemma~\ref{lem:tucker} to~$ K $. We define a weight function $ w\colon V\!\to\R $ by
	\[ w(v)\coloneqq\sum_{i\colon v\in A_i\cap B_i}x_i \,. \]
	Note that $ w $ has its image in $ \R_{\ge 0} $ and observe further that, for $ Y\sub V $, we have
	\[ w(Y)=\sum_{y\in Y}w(y)=\sum_{i=1}^{n}x_i\cdot\abs{Y\cap(A_i\cap B_i)}. \]
	With this, for $ i\le n $, we have
	\begin{align*}
	w(B_i)-w(A_i)&=\sum_{j=1}^nx_j\cdot\braces{\abs{B_i\cap(A_j\cap B_j)}-\abs{A_i\cap(A_j\cap B_j)}}\\
	&=\sum_{j=1}^{n}x_j\cdot m_{ij}\\
	&=(Mx)_i\,,
	\end{align*}
	where $ (Mx)_i $ denotes the $ i $-th coordinate of $ Mx $. Thus $ w $ is the desired weight function if we can show that $ Mx>0 $, that is, if all entries of $ Mx $ are positive.
	
	From $ x+Kx>0 $ we know that at least one entry of $ x $ is positive. Let us first consider the case that $ x $ has two or more positive entries. Then $ K'x>0 $ since~$ K' $ has positive values everywhere but on the diagonal, and hence
	\[ Mx=(K+K')\,x>0 \]
	since $ Kx\ge 0 $. Therefore, in this case, $ w $ is the desired weight function.
	
	Consider now the case that exactly one entry of $ x $, say $ x_i $, is positive, and that $ x $ is zero in all other coordinates. Then for $ j\ne i $ we have $ (Mx)_j\ge(K'x)_j>0 $ and thus $ w(B_j)-w(A_j)=(Mx)_j>0 $. However $ (Mx)_i=0 $ and thus $ w(A_i)=w(B_i) $, so $ w $ is not yet as claimed. To finish the proof it remains to modify $ w $ such that $ w(A_i)<w(B_i) $ while ensuring that we still have $ w(A_j)<w(B_j) $ for $ j\ne i $. This can be achieved by picking a sufficiently small $ \varepsilon>0 $ such that $ {w(A_j)+\varepsilon<w(B_j)} $ for all $ j\ne i $, picking any $ {v\in B_i\sm A_i }$, and increasing the value of $ w(v) $ by~$ \varepsilon $.	
\end{proof}

We conclude this section with the remark that Theorem~\ref{thm:tangledecider} and its proof extend to tangles in hypergraphs without any changes. Even more generally, the following version of Theorem~\ref{thm:tangledecider}, which is formulated in the language of~\cite{ProfilesNew}, can be established with exactly the same proof as well:

\begin{THM}\label{thm:profiledecider}
    Let $ \vec{U} $ be a universe of set separations of a finite ground-set $ V $ with the order function $ \abs{(A,B)}\coloneqq\abs{A\cap B} $. Then for any regular $ k $-profile $ P $ in $ \vec{U} $ there exists a function $ w\colon V\!\to\N $ such that a separation $ (A,B) $ of order $ {<k} $ lies in $ P $ if and only if~$ w(A)<w(B) $.
\end{THM}

In~\cref{thm:profiledecider} a set separation of some ground-set $ V $ is a pair $ (A,B) $ of subsets of~$ V $ with~$ A\cup B=V $. A set $ \vec{U} $ of such separations is a universe if $ \vec{U} $ contains $ (B,A) $ and~$ (A\cup C\,,\,B\cap D) $ for all $ (A,B) $ and $ (C,D) $ in~$ \vec{U} $. As for graphs, a partial order on the set separations of $ V $ is given by letting $ (A,B)\le(C,D) $ if $ A\cup C $ and~$ B\supseteq D $.

For an integer $ k $, a regular $ k $-profile in $ \vec{U} $ is a set $ P $ consisting of exactly one of $ (A,B) $ and $ (B,A) $ for every $ (A,B) $ in $ \vec{U} $ of order $ \abs{A\cap B} < k $, with the additional property that there are no $ (A,B) $ and $ (C,D) $ in $ P $ for which $ (B,A)\le(C,D) $ or such that~$ P $ contains~$ (B\cap D\,,\,A\cup C) $.

Observe that if $ G=(V,E) $ is a (hyper-)graph then the set $ \vec{U} $ of all separations of $ G $ is such a universe. Moreover every $ k $-tangle $ \tau $ of $ G $ is also a regular $ k $-profile of~$ \vec{U} $. (See~\cite{ProfilesNew} for more on the relation between graph tangles and profiles.) Therefore~\cref{thm:profiledecider} indeed applies to tangles in graphs and hypergraphs as well.

Theorem~\ref{thm:profiledecider} holds with the same proof as Theorem~\ref{thm:tangledecider}, since Lemma~\ref{lem:cross-counting} holds in this setting too: the only difference being that to see that $ (B\cap D\,,\,A\cup C) $ cannot lie in the profile at hand one now has to use the definition of a regular $ k $-profile rather than the fact that $ A $, $ C $, and $ B\cap D $ cover~$ G $.

\section{Edge-tangles}\label{sec:edgetangles}

A related object of study (cf.~\cites{liu_packing_2015,diestel_tangle-tree_2017}) to the (vertex-)tangles discussed above are the edge-tangles of a graph. In this context one considers the \emph{(edge) cuts} of a (multi-)graph $ G = (V,E) $, i.e.\ bipartitions $ (A,B) $ of $ V $. The \emph{order} of a cut $ (A,B) $ is the number of edges in $ G $ that are incident with vertices of both $ A $ and~$ B $.
For an integer~$ k $, a \emph{$ k $-edge-tangle} of~$ G $ is a set~$ \tau $ consisting of exactly one $ (A,B) $ or $ (B,A) $ for every cut $ (A,B) $ of order~$ < k $, with the additional properties that~$ \tau $ has no subset $ \menge{(A_1,B_1), (A_2,B_2), (A_3,B_3)} $ such that $ B_1\cap B_2\cap B_3=\emptyset $, and that~$ \tau $ contains no cut $ (A,B) $ for which~$ B $ is incident with fewer than~$ k $ edges of~$ G $.

In very much the same way as above we can prove the following theorem:

\begin{THM}\label{thm:edgetangledecider}
	Let $ G=(V,E) $ be a finite (multi-)graph and $ \tau $ a $ k $-edge-tangle in $ G $. Then there exists a function $ w\colon V\to\N $ such that a cut $ (A,B) $ of $ G $ of order $ <k $ lies in $ \tau $ if and only if~$ w(A)<w(B) $.
\end{THM}

We shall prove a more general version of this theorem where we allow $ G $ to be a graph with $ \R_{\ge 0} $\nobreakdash-weighted edges. We consider edges of weight $ 0 $ as indistinguishable from non-edges. Consequently, rather than a graph with weighted edges, we will just consider a pair $ (V, e) $ of a finite set $ V $ together with a symmetric function $ e \colon V^2 \to \R_{\ge 0} $, which we shall call a \emph{pairwise weighting} to distinguish it from the weight function of a decider. The \emph{order} of a bipartion $ (A,B) $ is defined as~$ |(A,B)| \coloneqq \sum_{(u,v)\in A\times B} e(u,v)$.
Note that this function is submodular in the sense that for all bipartitions $ (A,B) $ and $ (C,D) $ we have
\[ |(A,B)| + |(C,D)| \ge |(A \cup C\,,\, B \cap D)| + |(A \cap C\,,\, B \cup D)|\,.  \]

For any positive $ r  $ an \emph{$ r $-profile} in $ (V, e) $ is a set $ \tau $ consisting of exactly one of $ (A, B) $ or $ (B,A) $ for every bipartition $ (A,B) $ of $ V $ of order $ < r $, such that $ \tau $ does not contain $ (V,\emptyset) $ and has no subset of the form~$ \{(A,B), (C,D), (B\cap D, A\cup C)\} $.

Observe that every $ k $-edge-tangle of a (multi-)graph $ G=(V,E) $ is also a $ k $-profile in $ (V,e) $, where $ e $ is the multiplicity of the edges of~$ G $. Therefore the following theorem directly implies~\cref{thm:edgetangledecider}:

\begin{THM}\label{thm:weighting_decider}
	Let $ (V,e) $ be a pairwise weighting and $ \tau $ an $ r $-profile in $ (V,e) $.
	Then there exists a function $ w\colon V\to\N $ such that a bipartition $ (A,B) $ of $ V $ of order $ <r $ lies in $ \tau $ if and only if~$ w(A)<w(B) $.
\end{THM}

The main idea for proving this theorem is to first find an approprate weighting on the edges by the same principles as in \cref{thm:tangledecider} and to then transform it into the weighted vertex decider $w$. So let us first show an analogue of \cref{lem:cross-counting} for pairwise weightings. For this, we define a partial order on the bipartitions of $ V $ as in the previous section: by letting $ (A,B)\le(C,D) $ if and only if $ A\sub C $ (and thus $ B\supseteq D $). Using this partial order we can prove the following analogue of~\cref{lem:cross-counting}:

\begin{LEM}\label{lem:cross-counting_edges}
	For every $ r $-profile $ \tau $ in a pairwise weighting $ (V,e) $ and distinct maximal elements $ (A,B),(C,D) $ of $ \tau $ we have
 	\[ \sum_{(u,v)\in \mathrlap{B^2\,\cap\, (C\times D)}}\quad e(u,v) \;+ \sum_{(u,v)\in \mathrlap{D^2\,\cap\, (A\times B)}}\quad e(u,v) > 	\sum_{(u,v)\in \mathrlap{A^2\,\cap\, (C\times D)}}\quad e(u,v) \; + \sum_{(u,v)\in \mathrlap{C^2\,\cap\, (A\times B)}}\quad e(u,v)\,.\]
\end{LEM}

\begin{proof}
 The bipartition $ (A\cup C,\,B\cap D)$ of $ V $ is strictly larger in the partial order than the maximal elements $ (A,B) $ and $ (C,D) $ and hence cannot lie in $ \tau $. However, by the definition of an $r$-profile, $ \tau $ cannot contain $ (B\cap D,\,A\cup C) $ either. Thus we must have $ \abs{(A\cup C,\,B\cap D)}\ge r $, from which it follows by submodularity that $ \abs{(A\cap C,\,B\cup D)}<r $. Combining these two inequalities, using the definition of order and adding $\sum_{u\in A\cap C}\sum_{v\in B\cap D}e(u,v)$ to both sides proves the claim.
\end{proof}

We are now ready to prove~\cref{thm:weighting_decider}:

\begin{proof}[Proof of Theorem \ref{thm:weighting_decider}]
As in the proof of Theorem~\ref{thm:tangledecider}, it suffices to find a suitable real-valued weight function $ w\colon V\to\R_{\ge 0} $ since $ V $ is finite. We will begin by finding a weight function $\overline{w}\colon V^2\to \R_{\ge 0}$ on the pairs in $ V $ such that we have $\overline{w}(A)\le \overline{w}(B) $ for all $ (A,B)\in \tau$, where $\overline{w}(A) = \sum_{(u,v) \in A^2} \overline{w}(u,v)$, and with this inequality being strict for all but at most one of the maximal elements of $ \tau $. We will subsequently use this $ \overline{w} $ to construct the desired weight function~$w\colon V\to \R_{\ge 0}$.

Enumerate the maximal elements of $\tau$ as $(A_1,B_1),\dots,(A_n,B_n)$. Just as in~\cref{thm:tangledecider} it suffices to find a weight function which decides these maximal elements. For every two maximal elements $(A_i,B_i)$ and $(A_j,B_j)$ let
\[ m_{ij}\coloneqq\sum_{(u,v)\in \mathrlap{B_i^2 \;\cap\; (A_j\times B_j)}} \quad e(u,v) - \sum_{(u,v)\in \mathrlap{A_i^2 \,\cap\, (A_j\times B_j)}}\quad e(u,v)\,. \]
Let $ M $ be the matrix $ \{m_{ij}\}_{i,j\le n} $. Observe that, by Lemma~\ref{lem:cross-counting_edges}, $M + M^T$ has positive entries everywhere but on the diagonal, where it is zero. We are now in the same situation as in the proof of Theorem~\ref{thm:tangledecider} and can find some vector $x\in\R_{\ge 0}^n$ such that either $(Mx)_i>0$ on all $i$, or $x$ has exactly one non-zero entry, say $x_i$, and $(Mx)_j>0$ for all $j\neq i$.

In either case, given a pair of vertices $(u,v)$ let 
\[\overline{w}(u,v)\coloneqq e(u,v)\left(\sum_{j\colon (u,v) \in A_j\times B_j }\!x_j+\sum_{j\colon (u,v)\in B_j\times A_j} \!x_j\right) = \sum_{\overset{j:}{(u,v) \in \mathrlap{(A_j\times B_j )\cup (B_j\times A_j)}}} x_j \cdot e(u,v)\,.\]
Note that $\overline{w}$ is symmetric. For the same reason as in~\cref{thm:tangledecider}, by choice of $ x $, this function $ \overline{w} $ decides all but at most one of the $ (A_i,B_i) $ correctly in the sense that $ \overline{w}(A_i)\le\overline{w}(B_i) $ for all $ i=1,\dots,n $ with at most one inequality not being strict.

It remains to turn $ \overline{w} $ into a weight function on $ V $ rather than on $ V^2 $, and to verify that it has the desired properties. Define $w\colon V\to\R_{\ge 0}$ as
\[ w(v) \coloneqq \sum_{u\in V} \overline{w}(u,v)\,. \]
Then for each $ i=1,\dots,n $ we find that
\[\begin{split}
w(B_i) - w(A_i) 
 &= \sum_{u \in B_i}\sum_{v\in V} \overline{w}(u,v) - \sum_{u\in A_i}\sum_{v\in V} \overline{w}(u,v) \\
 &= \sum_{(u,v) \in B_i^2} \overline{w}(u,v) - \sum_{(u,v) \in A_i^2} \overline{w}(u,v) \\
 &= \sum_{(u,v)\in B_i^2} \sum_{\overset{j:}{(u,v) \in \mathrlap{(A_j\times B_j)\cup(B_j\times A_j)}}} x_j \cdot e(u,v) \qquad - \sum_{(u,v) \in A_i^2} \sum_{\overset{j:}{(u,v) \in \mathrlap{(A_j\times B_j )\cup(B_j\times A_j)}}} x_j \cdot e(u,v) \\
&= 2\sum_{j=1}^n \left( \sum_{(u,v)\in\mathrlap{B_i^2 \cap (A_j\times B_j)}} x_j \cdot e(u,v) \; - \sum_{(u,v) \in \mathrlap{A_i^2 \cap (A_j\times B_j)}} x_j \cdot e(u,v) \; \right) \\
&=2(Mx)_i.
\end{split}\]
Thus either $w(B_i)>w(A_i)$ for all maximal elements of $\tau$, from which the claim follows directly, or there is a single maximal element $(A_i,B_i)$ of $\tau$ such that $w(B_i)=w(A_i)$ and $w(B_j)>w(A_j)$ for all others. However, as in the proof of Theorem \ref{thm:tangledecider}, in the latter case we can pick an arbitrary vertex $v\in B_i$ and increase $w(v)$ by some small $\varepsilon>0$ to achieve $w(B_i)>w(A_i)$ while keeping $w(B_j)>w(A_j)$ for all other maximal elements of $\tau$.
\end{proof}

Remarkably, and in contrast to~\cref{thm:tangledecider},~\cref{thm:edgetangledecider} does not in fact extend to hypergraphs.
To see this, let us recall the relevant definitions, which extend naturally to hypergraphs.

A hypergraph $H=(V,E)$ consists of a vertex set $V$ together with a set $E\subseteq 2^V$ of hyperedges. An \emph{(edge) cut} of $ H $ is a bipartition $ (A,B) $ of~$ V $ and the \emph{order} of such an edge cut $ (A,B) $ is the number of hyperedges of~$ H $ that are incident with vertices from both~$ A $ and~$ B $.

For an integer~$ k $, a \emph{$ k $-edge-tangle} of~$ H $ is a set~$ \tau $ consisting of exactly one $ (A,B) $ or $ (B,A) $ for every cut $ (A,B) $ of order~$ < k $, with the additional properties that~$ \tau $ has no subset $ \menge{(A_1,B_1), (A_2,B_2), (A_3,B_3)} $ such that $ B_1\cap B_2\cap B_3=\emptyset $, and that~$ \tau $ contains no cut $ (A,B) $ for which~$ B $ is incident with fewer than~$ k $ hyperedges of~$ H $.

A weighted decider for some~$ k $-edge-tangle~$ \tau $ of a hypergraph $ H=(V,E) $ then is a function $ w\colon V\to\N $ such that a cut $ (A,B) $ of~$ H $ of order $ <k $ lies in~$ \tau $ if and only if~$ w(A)<w(B) $.

\cref{thm:edgetangledecider} thus asserts that if~$ H $ is just a (multi-)graph, i.e., if every hyperedge in~$ E $ has size~$ 2 $, then every~$ k $-edge-tangle of~$ H $ has such a weighted decider. We are now going to construct an example demonstrating that this may fail for hypergraphs~$ H $ with hyperedges of size $\ge 3$.

\begin{EX}\label{ex:construction}
	For some natural number $ k\ge 6 $ let $ \ell $ be an integer with~$ 3\le \ell\le\frac{k}{2} $. Let~$ V $ be the set of all $ \ell $-element subsets of~$ [k]=\menge{1,\dots,k} $. Let the set $ E $ of hyperedges consist of, for each $ i\in[k] $, the set of all $ v\in V $ that contain~$ i $.
	Note that each of these~$ k $ many hyperedges of~$ H $ has size $ \binom{k-1}{\ell-1} $, making~$ H $ a uniform $ \ell $-regular hypergraph.
\end{EX}


\begin{THM}\label{thm:edgeexampel}
	Let $H$ be as in \cref{ex:construction}. Then~$ H $ has a~$ k $-edge-tangle with no weighted decider.
\end{THM}

\begin{proof}
	Let $ S_k $ denote the set of all cuts of~$ H $ of order~$ <k $. For a set $ A\sub V $ we write $ \cup A $ for the set $ \bigcup_{v\in A}v $, which is a subset of~$ [k] $. Observe that for every cut $ (A,B) $ of~$ H $ at most one of $ \cup A $ and $ \cup B $ can be a proper subset of~$ [k] $. Note further that a cut $ (A,B) $ of~$ H $ lies in $ S_k $ if and only if at least one of the~$ k $ hyperedges of~$ H $ does not meet both~$ A $ and~$ B $, which is the case precisely if one of $ \cup A $ and $ \cup B $ is a proper subset of~$ [k] $.
	
	We can therefore define
	\[ \tau\coloneqq\menge{(A,B)\in S_k\mid \cup A\subsetneq[k]}\,. \]
	Let us show that~$ \tau $ is a~$ k $-edge-tangle of~$ H $ with no weighted decider.
	
	To see that~$ \tau $ is a~$ k $-edge-tangle we note that by the above observation~$ \tau $ contains exactly one of $ (A,B) $ or $ (B,A) $ for every cut~$ (A,B)\in S_k $. Furthermore if $ (A_1,B_1),(A_2,B_2),(A_3,B_3)\in\tau $ then any element of~$ V $ containing at least one point each from $ [k]\sm\cup A_1 $, from $ [k]\sm\cup A_2 $, and from $ [k]\sm\cup A_3 $ lies in $ B_1\cap B_2\cap B_3 $, which is hence non-empty since such a $ v\in V $ exists by~$ \ell\ge 3 $. Finally for each $ (A,B)\in\tau $ the set~$ B $ is incident with each hyperedge of~$ H $ since~$ \cup B=[k] $. Thus~$ \tau $ is indeed a~$ k $-edge-tangle.
	
	Finally, let us show that~$ \tau $ has no weighted decider. Suppose for a contradiction that some weighted decider $ w\colon V\to\N $ for~$ \tau $ exists. For each $ i\in[k] $ consider the cut $ (A_i,B_i) $, where
	\[ A_i\coloneqq\menge{v\in V\mid i\notin v}\qquad\textnormal{ and }\qquad B_i\coloneqq\menge{v\in V\mid i\in v}\,, \]
	and note that $ (A_i,B_i)\in\tau $. Since $ w $ is a weighted decider for~$ \tau $ we have $ w(B_i)>w(A_i) $ for each $ i\in[k] $. We therefore have
	\[ \sum_{i\in[k]}(w(B_i)-w(A_i))>0\,, \]
	since each term in the sum is positive. By counting the instances of $ w(v) $ occurring in the sum for each $ v\in V $ we find that 
	\[ \sum_{i\in[k]}(w(B_i)-w(A_i))=\sum_{v\in V}w(v)\cdot\braces{\abs{\menge{i\in[k]\mid i\in v}}-\abs{\menge{i\in[k]\mid i\notin v}}}\,, \]
	since $ v\in B_i $ if and only if $ i\in v $, and otherwise~$ v\in A_i $. The left-hand side of this equation is positive. However, in contradiction to this, no term of the right-hand sum is greater than zero since we have by $ \ell\le\frac{k}{2} $ that
	\[ \abs{\menge{i\in[k]\mid i\in v}}-\abs{\menge{i\in[k]\mid i\notin v}}=\ell-(k-\ell)\le 0\,. \]
	Therefore there can be no weighted decider for~$ \tau $.
\end{proof}

A construction analogous to~\cref{ex:construction} was found independently by Geelen~\cite{GeelenExample} in the setting of matroids, who used it to show that matroids, too, can have tangles with no weighted decider.

Finally, let us remark that~\cref{ex:construction} can also be used to show that allowing weighted deciders to take values in $ \mathbb{R} $ rather than $ \N $ does not suffice to guarantee their existence for edge-tangles of hypergraphs: for $ k=2\ell $~the tangle described in~\cref{thm:edgeexampel} has no weighted decider with real-valued and possibly negative weights either, with the same proof.



\begin{aicauthors}
\begin{authorinfo}[ce]
  Christian Elbracht\\
  Universit\"at Hamburg\\
  Hamburg, Germany\\
  christian.elbracht\imageat{}uni-hamburg\imagedot{}de
\end{authorinfo}
\begin{authorinfo}[jk]
  Jay Lilian Kneip\\
  Universit\"at Hamburg\\
  Hamburg, Germany\\
  jkneip.math.uhh\imageat{}gmail\imagedot{}com
\end{authorinfo}
\begin{authorinfo}[mt]
  Maximilian Teegen\\
  Universit\"at Hamburg\\
  Hamburg, Germany\\
  maximilian.teegen\imageat{}uni-hamburg\imagedot{}de
\end{authorinfo}
\end{aicauthors}

\end{document}